\documentclass{amsart}
\usepackage{latexsym}
\usepackage{amsmath}
\usepackage{amssymb}
\usepackage{amsthm}
\usepackage{amsfonts}
\usepackage{bbm}
\usepackage{bbding}
\usepackage{cite}
\usepackage{dcolumn}
\usepackage{indentfirst}
\usepackage{mathrsfs}
\usepackage{pifont}

\newtheorem{theorem}{Theorem}[section]
\newtheorem{lemma}[theorem]{Lemma}
\newtheorem{corollary}[theorem]{Corollary}

\theoremstyle{definition}
\newtheorem{definition}[theorem]{Definition}

\theoremstyle{remark}
\newtheorem{remark}[theorem]{Remark}

\numberwithin{equation}{section}

\begin{document}

\title{Singularities of mean curvature flow and isoperimetric inequalities in $\mathbb{H}^3$}

\author{Kui Wang}
%    Address of record for the research reported here
\address{School of Mathematic Sciences, Fudan University, ShangHai, 200433}
%    Current address
\curraddr{Department of Mathematics, University of California, San Diego,
 La Jolla , CA 92093}
\email{09110180001@fudan.edu.cn}
%    \thanks will become a 1st page footnote.
\thanks{The author was sponsored by the China Scholarship Council for two year
study at University of California, San Diego.}

%    Information for second author
%\author{Author Two}
%\address{Mathematical Research Section, School of Mathematical Sciences,
%Australian National University, Canberra ACT 2601, Australia}
%\email{two@maths.univ.edu.au}
%\thanks{Support information for the second author.}
%\subjclass[2000]{35K55, 35K45, 58J35.}

%\date{January 1, 2001 and, in revised form, June 22, 2001.}

%\dedicatory{This paper is dedicated to our advisors.}

\keywords{Mean curvature flow, Isoperimetric inequality, Willmore energy, Hyperbolic space}

\begin{abstract}
In this article, by following the method in \cite{PT}, combining  Willmore energy
with isoperimetric inequalities, we construct two
examples of singularities under mean curvature flow in $\mathbb{H}^3$. More precisely,
there exists a torus, which must develop a
singularity under MCF before the volume it encloses decreases to zero.
There also exists a topological sphere in the shape of a dumbbell,
which must develop a
singularity in the flow before its area shrinks to zero.
Simultaneously, by using the flow, we proved an isoperimetric
inequality for some domains in $\mathbb{H}^3$.
\end{abstract}

\maketitle

\section{Introduction}
Let $N$ be a smooth $n-1$ dimensional compact manifold without boundary and
$F_0: N \rightarrow M^n$ be a smooth embedding into an $n$-dimensional Riemannian manifold $(M,g)$.
An evolution $F(x,t):N\times [0,T)\rightarrow M$ is defined as mean curvature flow (MCF)
with initial hypersurface $F_0(N)$ if it is satisfying
\begin{equation}
\left\{\begin{array}{l}
\large{\partial_t F(x,t)=-H\nu}\\
\large{F(x,0)=F_0}\\
\end{array},\right.\label{mcf}
\end{equation}
where $H$ denotes the mean curvature with respect to outer unit normal direction $\nu$.
This flow has been a useful topic in the study of geometry problems
and there are many good results (cf.\cite{GH,XZ,CM}). One useful application is to study the isoperimetric
inequality. The classical isoperimetric inequality in $\mathbb{R}^{n}$ states that
for any bounded domain $\Omega$ in $\mathbb{R}^{n}$ with smooth boundary $\partial \Omega$, we have
\begin{equation}
Area(\partial \Omega)\geq n\omega_n^{\frac{1}{n}}Vol^{\frac{n-1}{n}}(\Omega)\label{ISO},
\end{equation}
where $\omega_n$ is the volume of the unit ball in $n$ dimensional Euclidean space.
Isoperimetric problem on manifolds with non-positive sectional curvature is still open for
higher dimensions. For dimension 3 and 4, the isoperimetric inequality
 was proved  by B. Kleiner in \cite{BK} and  B. Croke in \cite{BC} respectively.

On one hand, isoperimetric inequalities can be derived from curvature flows somehow, especially in $\mathbb{R}^n$
(e.g. \cite{GH1,PT,EV, FS}). In \cite{GH1}, G. Huisken explained that MCF can be expected to converge to
solutions of the isoperimetric problem in $\mathbb{R}^n$. In \cite{EV}, the isoperimetric inequality was
proved by MCF in hyperbolic spaces for some special domains.  On the other hand, we can use the isoperimetric
inequality to study the singularities of the flow. In $\mathbb{R}^3$, P. Topping used the isoperimetric inequalities to study
singularities and in \cite{PT} two examples of singularities are developed.
One is a torus and the other is a topological sphere.

Recently, B. Andrews, Chen and etc have studied curvature flows in hyperbolic spaces in\cite{AC, AHLW}.
 They proved in \cite{AC} that MCF with initially Gauss curvature $GK>1$ pointwisely in $\mathbb{H}^3$,
the flow $\Sigma_t$ shrinks to a point and becomes more spherical. Here
$GK$ denotes the Gauss-Kronecker curvature  which is defined by the product
of the principle curvatures of $\Sigma_0$ in $\mathbb{H}^3$.
One motivation of this paper is to obtain the isoperimetric inequality by the flow.
The other motivation is to construct singular examples of the flow which violate the condition
$GK>1$.

This paper is organized as follows. In Section 2, we study the properties of
Willmore energy in three  dimensional hyperbolic space. In Section 3,
 we establish a precise comparison theorem for MCF in $\mathbb{H}^n$.
In Section 4, we use the mean curvature flow to prove the isoperimetric inequality
in $\mathbb{H}^3$. In Section 5 and 6, we give two examples of singularities
evolved by mean curvature flow.

\section{Willmore energy}
\begin{definition}
For any closed (compact without boundary) hypersurface $\Sigma\subset \mathbb{R}^3$,
the Willmore energy is defined by
\begin{equation*}
W(\Sigma)\doteqdot\int_{\Sigma} H^2 d\sigma,
\end{equation*}
where $d\sigma$ denotes the area element of $\Sigma$ and $H=\frac{k_1+k_2}{2}$
($k_1,k_2$ are the principal curvatures of $\Sigma$).
\end{definition}

For any immersed closed hypersurface, the lower bound of Willmore energy was obtained by Li-Yau
in \cite{LY} and P. Topping in \cite{PT}. More precisely,
\begin{equation}
W(\Sigma)=\int_{\Sigma} H^2 d\sigma\geq\int_{\Sigma^{+}} GK d\sigma\geq4\pi, \label{W5}
\end{equation}
where $\Sigma^{+}=\{x\in \Sigma, GK(x)\geq 0\}$ and
the equality holds if only if $\Sigma$ is a round sphere.

The problem of minimizing the Willmore energy among the class of immersed tori was proposed
by Willmore and in 1965 he conjectured in \cite{TW1} that the Willmore energy of a torus immersed 
in $\mathbb{R}^3$ is at least $2\pi^2$. The existence of a torus that minimizes the Willmore 
energy was established by L. Simon in \cite{LM}. Together with (\ref{W5}), for any torus $\Sigma$,
we have the following estimate
\begin{equation}
\int_{\Sigma} H^2 d\sigma\geq c_0>4\pi, \label{W6}
\end{equation}
where $c_0$ is a constant.  
Recently, F. Marques and A. Neves proved this conjecture. For further details, one can refer to the paper \cite{FA}.

In hyperbolic space $\mathbb{H}^3$, we consider the following energy instead of Willmore energy in $\mathbb{R}^3$
\begin{equation}
\bar{W}(\Sigma)\doteqdot\int_{\Sigma} (H^2-1) d\sigma.\label{W1}
\end{equation}
This quantity comes from the conformal transformation from $\mathbb{R}^3$ to $\mathbb{H}^3$.
Further more, we have the following conformal invariant.
\begin{lemma}
Assume $(M^3, g)$ be a Riemannian manifold and $\Sigma$ be a compact hypersurface in $M$.
Let $\lambda$ be a positive function  defined on $M$, and $\bar{g}\doteqdot \lambda^2 g$ be
a conformal metric. Denote by $K$ (or $\bar{K}$) the sectional curvature of $\Sigma$ with respect to
the metrics $g$ (or $\bar{g}$). Then
\begin{equation}
\int_{\Sigma} H^2+KdA=\int_{\Sigma} \bar{H}^2+\bar{K}d\bar{A}, \label{C1}
\end{equation}
where $H$ (or $\bar{H}$) denotes the mean curvature of $\Sigma$ with respect to the outer normal
direction and $dA$ (or $d\bar{A}$) denotes the area element with respect to metric $g$ (or $\bar{g}$).
\end{lemma}
The proof of this lemma can be done by direct calculations and the Gauss-Bonnet formula.
For further details, one can
refer to the Chapter 7 in \cite{TW}.
By (\ref{W5}), (\ref{W6}) and the above Lemma, we conclude that
\begin{theorem}\label{WT1}
Let $\Sigma$ be any immersed closed hypersurface in $\mathbb{H}^3$, then we have
\begin{equation}
\int_{\Sigma} (H^2-1) d\sigma \geq 4\pi, \label{W3}
\end{equation}
and moreover if $\Sigma$ is any immersed torus in $\mathbb{H}^3$, we have
\begin{equation}
\int_{\Sigma} (H^2-1) d\sigma\geq c_0,\label{W2}
\end{equation}
where $c_0$ is a constant satisfying $c_0>4\pi$ .
\end{theorem}
\begin{remark}
For the inequality (\ref{W3}), one can see an alternative proof in \cite{MR}. By the result in \cite{FA}, the
constant $c_0$ is actually $2\pi^2$. But in this paper, we only need $c_0>4\pi$.
\end{remark}
\section{Maximum principles for MCF in $\mathbb{H}^n$}
\setcounter{equation}{0}
Let us begin with a comparison
principle for mean curvature flow. Roughly speaking, two nonintersecting hypersurfaces in $\mathbb{H}^3$
will remain nonintersecting when each is evolved simultaneously.
Let $\Sigma^{1}_t$ and $\Sigma^{2}_t$ be two hypersurfaces in $\mathbb{H}^n$ envolved by MCF
for $t\in[0,T)$, and assume that $\Sigma^{1}_0$ is compact.

Define the distance function between $\Sigma^{1}_t$ and $\Sigma^{2}_t$ by
\begin{equation}
d(t)\doteqdot \inf_{p\in \Sigma^{1}_t, q\in \Sigma^{2}_t} d(p,q),\label{F2}
\end{equation}
where $d(p,q)$ denotes the distance function in the standard hyperbolic spaces.

For any $t\in[0,T)$, $d(t)$ is locally Lipschitz in time, as the mean curvature is locally bounded
and the two hypersurfaces evolved by mean curvature.
 The infimum in (\ref{F2}) is actually obtained by some points since $\Sigma^{1}_t$ is
compact. Similarly as in $\mathbb{R}^n$, we establish a comparison principle in $\mathbb{H}^n$.
\begin{theorem}\label{MP}
$\Sigma^{1}_t$ and $\Sigma^{2}_t$ are the same as above. Then for any $t\in [0,T)$,
the distance function $d(t)$ satisfies
\begin{equation}
e^t\sinh{\frac{d(t)}{2}}\geq \sinh {\frac{d(0)}{2}}.\label{F1}
\end{equation}
In particular, we have
\begin{equation}
e^td(t)\geq d(0).\label{Fa1}
\end{equation}
\end{theorem}
\begin{proof} Since $d(t)$ is locally Lipschitz in time, we can assume that $t$ is
a differentiability point. We assume $d(t)>0$ and ($p_t, q_t$) be any pair realizing such
a minimum and $\gamma(s)$ be the normal geodesic from $p_t$ to $q_t$ for $s\in[0,d(t)]$.
Denote by $H_{p_t}$ (or $H_{q_t}$) mean curvature of $\Sigma^{1}_t$ (or $\Sigma^{2}_t$)
at $p_t$ (or $q_t$) with respect to outer normal
direction $\nu(p_t)$ (or $\nu(q_t)$).

By the first variation, we know that $\nu(p_t)\perp T_{p_t}\Sigma_t^{1}$ and
$\nu(q_t)\perp T_{q_t}\Sigma_t^{2}$. Therefore, without loss of generality we
assume that $\nu(p_t)=\gamma^{'}(p_t)$ and $\nu(q_t)=-\gamma^{'}(q_t)$.

Choose an orthogonal
basis $\{e_1,\cdots, e_{n-1}\}$ of $\Sigma^{1}_t$ at $p_t$ and parallel translate them along $\gamma$,
and then we get an orthogonal frame $\{e_1(s),\cdots, e_{n-1}(s), \gamma^{'}(s)\}$ at $\gamma(s)$.

For $k=1,2$, let $\gamma_k (u)$ be a curve in $\Sigma^{k}_t$ for $u\in(-\epsilon, \epsilon)$ satisfying
 $\gamma_1 (0)=p_t$,  $\gamma_2 (0)=q_t$, $\gamma_1^{'} (0)=e_i(0)$ and $ \gamma_2^{'} (0)=e_i(d)$.
Then we can get a family of geodesics from $\gamma_1(u)$ to $\gamma_2(u)$ which has the length of $L(u)$.
By the Jacobian equation, we can write the variation fields  as $U_i(s)=f(s)e_i(s)$ for $i=1,\cdots, n-1$,
where $f(s)$ is decided by
\begin{equation*}
\left\{\begin{array}{l}
\large{f^{''}(s)-f(s)=0}\\
\large{f(0)=1, f(d)=1}
\end{array},\right.
\end{equation*}
so we have
\begin{equation*}
f(s)=\cosh s-\tanh \frac{d}{2}\sinh s.
\end{equation*}
By the definition of $d(t)$, we have $L(u)\geq L(0)$ for $u\in(-\epsilon, \epsilon)$.
Thus the second variation formula gives that
\begin{eqnarray*}
0\leq\frac{d^2}{du^2}L(u)|_{u=0}=\langle \nabla_{U_i}{U_i}, \gamma^{'}\rangle|_0^d+\int_0^d|\dot{U}_i|^2-\langle R(\gamma^{'}, U_i)\gamma^{'}, U_i\rangle d s .
\end{eqnarray*}
Summing over $i$ from $1$ to $n-1$ gives
\begin{eqnarray*}
0&\leq&H_{p_t}+H_{q_t}+\int_0^d |\dot{f}|^2+f^2ds\\
&=&H_{p_t}+H_{q_t}+\dot{f}(d)-\dot{f}(0)\\
&=&H_{p_t}+H_{q_t}+2\tanh \frac{d}{2}.\\
\end{eqnarray*}
By Hamilton's trick in \cite{RH}, we know at differentiable point $t$,
\begin{equation*}
\frac{d}{dt}d(t)=\frac{d}{dt}d(p_t,q_t)=H_{p_t}+H_{q_t}\geq-2\tanh \frac{d}{2},
\end{equation*}
which gives
\begin{equation*}
e^t\sinh{\frac{d(t)}{2}}\geq \sinh {\frac{d(0)}{2}}.
\end{equation*}
for $t\in[0,T)$ provided by $d(0)>0$. If $d(0)=0$, then (\ref{F1}) is trival.
\end{proof}
\begin{remark}
In $\mathbb{R}^n$, the distance of hypersurface evolved by MCF satisfies $d(t)\geq d(0)$
(cf. in \cite{CM}). But the exponential term $e^t$ in (\ref{F1}) actually comes in because of
the curvature of the hyperbolic spaces.
\end{remark}
From the above theorem, we immediately have the following well known corollary which is also stated in \cite{XZ}.
\begin{corollary}(cf. Proposition 10.4 in \cite{XZ}) \label{C3}
Two hypersurfaces $\Sigma^{1}_t$ and $\Sigma^{2}_t$ evolved by MCF in $\mathbb{H}^n$
are disjoint for $t\in [0,T)$ provided
by initially hypersurfaces $\Sigma^{1}_0$ and $\Sigma^{2}_0$ are disjoint.
In particular, if $\Sigma^{1}_0$ is contained in the domain enclosed by
$\Sigma^{2}_0$ strictly (that is to say the distance between them is positive), then $\Sigma^{1}_t$ is strictly contained
in the domain enclosed by  $\Sigma^{2}_t$ for all $t\in [0,T)$.
\end{corollary}
\begin{remark}
As we shall see, the Corollary \ref{C3} may be used to restrict the movement of
an evolving hypersurface, by comparing it to a known solution of the mean
curvature flow (such as the flow of round spheres) which is initially disjoint.
We need to use this trick later on.
\end{remark}
For any compact hypersurface $\Sigma$, we define the diameter by
\begin{equation}
diam(\Sigma)\doteqdot \max_{p, q\in \Sigma} d(p,q).\label{F3}
\end{equation}
 The following is the
consequence of Corollary \ref{C3}.
\begin{corollary}
If compact hypersurfaces $\Sigma_t$ evolved by MCF in $\mathbb{H}^n$ for $t\in [0,T)$ and
let $T_{max}$ is the maximal time of smooth existence of the flow, then
\begin{equation}
T_{max}\leq \ln (\cosh{(diam(\Sigma_0))}.\label{F4}
\end{equation}
\end{corollary}
\begin{proof}
We can choose geodesic ball $B(r_0)$ in $\mathbb{H}^n$  such that $\Sigma_0$ is strictly
contained in $B(r_0)$ provided by $r_0>diam(\Sigma_0)$. The MCF equation of geodesic balls gives
\begin{equation}
\frac{d}{dt}r(t)=-coth r(t), \label{F5}
\end{equation}
with $r(0)=r_0$.

Solving the above ODE, we get
\begin{equation}
\cosh{r(t)}=e^{-t}\cosh{r_0}. \label{F6}
\end{equation}
By the Corollary \ref{C3} and equation (\ref{F6}), we conclude
\begin{equation*}
T_{max}\leq \ln (\cosh{r_0}).
\end{equation*}
Let $r_0\rightarrow diam(\Sigma_0)$, we get the estimate (\ref{F4}).
\end{proof}

\section{An isoperimetric inequality in $\mathbb{H}^3$}
\setcounter{equation}{0}
In this section we will use Willmore energy to explain the relationships between MCF
and isoperimetric inequalities in $\mathbb{H}^3$. Further more, we get an isoperimetric inequality
for some domains.

To begin with, we give some notations.
$A(t)$ denotes the area of $\Sigma_t$ and $V(t)$ denotes the
volume of the domain bounded by $\Sigma_t$. Under mean curvature flow,
the evolution equations of $A(t)$ and $V(t)$ are given by
\begin{equation}
\frac{d}{dt}A(t)=-2\int_{\Sigma_t} H^2 d\sigma_t,\label{M2}
\end{equation}
and
\begin{equation}
\frac{d}{dt}V(t)=-\int_{\Sigma_t} H d\sigma_t.\label{M3}
\end{equation}
We get the following isoperimetric inequality.
\begin{theorem}
Let $\Sigma_0$ be a smooth closed hypersurface in $\mathbb{H}^3$. If $\Sigma_0$ under MCF
has no singularity before the volume it encloses shrinks to zero, then
\begin{equation}
Area(\Sigma_0)\geq Area(\bar{\Sigma}),\label{M10}
\end{equation}
where $\bar{\Sigma}$ is a geodesic sphere with the volume it encloses is the same as
the volume $\Sigma_0$ encloses.
Moreover, if the equality holds, then $\Sigma_0$ is a geodesic sphere.
\end{theorem}
\begin{proof}
By (\ref{M3}) and the Cauchy-Schwartz inequality, we have
\begin{equation}
-\frac{d}{dt}V(t)\leq\big(\int_{\Sigma_t}H^2d\sigma_t\big)^{\frac{1}{2}}A^{\frac{1}{2}}(t).\label{e1}
\end{equation}
The inequality (\ref{W2}) yields that
\begin{equation}
\int_{\Sigma_t}H^2d\sigma\geq A(t)+4\pi,\label{e2}
\end{equation}
thus we have
\begin{equation*}
-\frac{d}{dt}V(t)\leq\big(\int_{\Sigma_t}H^2d\sigma\big)^{\frac{1}{2}}A^{\frac{1}{2}}(t)(\frac{\int_{\Sigma_t}H^2d\sigma
}{4\pi+A(t)})^{\frac{1}{2}}
=-\frac{1}{2}\frac{d}{dt}A(t)(\frac{A(t)}{4\pi+A(t)})^{\frac{1}{2}}.\\
\end{equation*}
Integrating the above inequality from $t=0$ to $t=T$ (where $T$ is the time as the
 volume of $\Sigma_T$ encloses shrinking to zero), we obtain
\begin{equation}
V_0\leq\frac{1}{2}\int_{A(T)}^{A_0}(\frac{x}{4\pi+x})^{\frac{1}{2}}dx,\label{M4}
\end{equation}
where $A_0$, $V_0$ denote the area of the initial hypersurface $\Sigma_0$ and the volume of
the domain it encloses respectively. Therefore inequality (\ref{M4}) implies
the isoperimetric inequality (\ref{M10}).

In fact, if we assume $r_0$ satisfies $V_0=4\pi\int_0^{r_0}\sinh^2(t)dt$,
then we know that $Area(\bar{\Sigma})=4\pi\sinh^2(r_0)$.
Let $s=4\pi\sinh^2(t)$, and direct calculation shows that
\begin{equation*}
V_0=\frac{1}{2}\int_{0}^{Area(\bar{\Sigma})}(\frac{s}{4\pi+s})^{\frac{1}{2}}ds.
\end{equation*}
Using the above equation and (\ref{M4}), we conclude
\[\frac{1}{2}\int_{0}^{Area(\bar{\Sigma})}(\frac{x}{4\pi+x})^{\frac{1}{2}}dx
\leq\frac{1}{2}\int_{A(T)}^{A_0}(\frac{x}{4\pi+x})^{\frac{1}{2}}dx,\]
so (\ref{M10}) holds.

If equality holds, from (\ref{e1}), we know that the mean curvature $H$ of $\Sigma_0$ is a constant,
and by (\ref{e2}) we know $H>1$.  For the same reason as in Section 4.4 of \cite{MR}, we deduce
that $\Sigma_0$ is a geodesic sphere.
\end{proof}

\section{ Singularities of 2-tori}
It is well known that using the mean curvature flow to
prove isoperimetric inequalities is always limited by the possibility of singularity formation
in flows.
Recall that the condition in \cite{AC} is $GK>1$ for initial hypersurface $\Sigma_0$ which implies
\begin{equation}
\int_{\Sigma_0}(GK-1)d\sigma>0.\label{CON}
\end{equation}
If $\Sigma_0$ is a torus, then the Gauss-Bonnet formula gives
\begin{equation*}
\int_{\Sigma_0}(GK-1)d\sigma=0,
\end{equation*}
which violates the condition (\ref{CON}). Our aim in this section is to give an example of
tori which must develop a singularity in the flow.
\begin{theorem}
Let $\Sigma_{\epsilon}$ be the boundary of a round unit ball $B_0$ in $\mathbb{H}^3$ with a thin hole (radius $\epsilon$) drilled
through it. Smooth $\Sigma_{\epsilon}$, and still denote by $\Sigma_{\epsilon}$.
Then there exists some $\epsilon_0$ small enough such that
$\Sigma_{\epsilon_0}$  under MCF must develop a singularity before the volume it
 encloses shrinks to zero.
\end{theorem}
\begin{proof}
For any small $\epsilon$, $\Sigma_{\epsilon}$ is a torus,
then (\ref{W2}) implies
\begin{equation*}
\int_{\Sigma_{\epsilon}}H^2 d\sigma \geq c_0+Area(\Sigma_{\epsilon}),
\end{equation*}
where $c_0>4\pi$ is a constant which is appearing in (\ref{W2}).

Suppose there is no singularity before the volume it encloses shrinks to zero. Following the
same process in Section 4,  we obtain
\begin{equation*}
|B_{\epsilon}|\leq\frac{1}{2}\int_{0}^{Area(\Sigma_{\epsilon})}(\frac{x}{c_0+x})^{\frac{1}{2}}
dx\leq\frac{1}{2}\int_{0}^{Area(\Sigma_{\epsilon})}(\frac{x}{4\pi+x})^{\frac{1}{2}}dx-c,
\end{equation*}
where \[c\doteqdot\frac{1}{2}\int_{0}^{2\pi}(\frac{x}{4\pi+x})^{\frac{1}{2}}-
(\frac{x}{c_0+x})^{\frac{1}{2}}dx\] is a positive constant independent of $\epsilon$.
Here we used the fact that $Area(\Sigma_{\epsilon})>2\pi$ provided by $\epsilon$ sufficiently small.

Let $\epsilon\rightarrow0$, we conclude
\begin{equation*}
Vol(B_0)\leq Vol(B_0)-c,
\end{equation*}
which is a contraction and therefore the MCF must develop a singularity before
the the volume it encloses shrinks to zero .
\end{proof}

\section{ Singularities of 2-spheres}
\setcounter{equation}{0}
In this section, we aim to construct a topological sphere which must develop a singularity under mean
curvature flow before its area shrinks to zero.  More precisely, we have the following
theorem.
\begin{theorem}\label{THE}
There exists an embedded topological sphere $M_0$ in the shape of a dumbbell in $\mathbb{H}^3$  enclosing two unit spheres
with centres separated by a sufficiently large distance $d$,
such that $M_0$ must develop a singularity under MCF before the area
of $M_t$ decreases to zero.
\end{theorem}
\begin{remark}
If $M_0$ is a topological sphere, then the Gauss-Bonnet formula gives
\begin{equation*}
\int_{M_0}(GK-1)d\sigma=4\pi>0.
\end{equation*}
which satisfies the condition (\ref{CON}). Thus the condition $GK>1$ in \cite{AC} can not be
replaced by (\ref{CON}).
\end{remark}
To prove Theorem \ref{THE}, we need an estimate of external diameter of closed hypersurfaces
in terms of its area and Willmore energy in $\mathbb{H}^3$. Here external diameter is defined by (\ref{F3}).

The estimate in Euclidean space was obtained  by L. Simon in  \cite{LM} with some universal
constant, and an other proof was given by P. Topping in \cite{PT, PT1}.
By following P. Topping's strategy, we prove a similar estimate in three dimensional hyperbolic space.
\begin{lemma}
For any closed hypersurface $M$ in $\mathbb{H}^3$, we have the estimate
\begin{equation}
diam(M)\leq\frac{7}{2\pi}A^{\frac{1}{2}}(M)W^{\frac{1}{2}}(M),\label{M5}
\end{equation}
where $W(M)=\int_M H^2 d\sigma$ and $A(M)$ denotes the area of $M$.
\end{lemma}
\begin{remark}
In \cite{PT}, P. Topping proved that in $\mathbb{R}^3$
\begin{equation*}
diam(M)<\frac{2}{\pi}A^{\frac{1}{2}}(M)W^{\frac{1}{2}}(M).
\end{equation*}
\end{remark}

\begin{proof} For any fixed point $x_0\in M\subset \mathbb{H}^3$, $0<\rho<\rho_0$, we denote
$d(x)\doteqdot d(x,x_0)$, $X\doteqdot\nabla \cosh d$ and $|X|_{\rho}=\max\{|X|, \rho\}$.

Define a
vector field on $\mathbb{H}^3$ by
\begin{equation*}
\Phi(x)\doteqdot(\frac{1}{|X(x)|_{\rho}^{2}}-\frac{1}{\rho_0^{2}})_{+}X(x).
\end{equation*}
Choose  an orthogonal normal frame $\{e_1, e_{2}\}$ of $TM$, then direct calculation shows
\begin{equation}
\sum_{i=1}^{2}\int_M\langle \nabla_{e_i}\Phi, e_i\rangle d
\sigma_M=2\int_M\langle \Phi, H\nu\rangle d\sigma_M, \label{L1}
\end{equation}
where $\nu$ is the outer normal direction of $M$.

In fact
\begin{eqnarray*}
\sum_{i=1}^{2}\int_M\langle \nabla_{e_i}\Phi, e_i\rangle d\sigma_M&=& \sum_{i=1}^{2}
\int_M e_i\langle\Phi, e_i\rangle-\langle \Phi, \nabla_{e_i}e_i\rangle d\sigma_M\\
&=& \sum_{i=1}^{2}\int_M e_i\langle\Phi, e_i\rangle-\langle
\Phi^{\perp}, \nabla_{e_i}e_i\rangle-\langle \Phi^{T}, \nabla^M_{e_i}e_i\rangle d\sigma_M\\
&=& \int_M div^M\Phi^{T}+2 \langle \Phi, H\nu\rangle d\sigma_M\\
&=& 2\int_M \langle \Phi, H\nu\rangle d\sigma_M,
\end{eqnarray*}
where $\nabla$ denotes the covariant derivative in $\mathbb{H}^3$.

Observing that in standard hyperbolic spaces, we have
\[\nabla^2 \cosh(d)=\cosh{d}g_{\mathbb{H}^3},\]
which implies
\[\nabla_{e_i}X=\cosh{d(x)}e_i.\]
Direct calculation shows that
\begin{equation*}
\sum_{i=1}^{2}\langle \nabla_{e_i}\Phi, e_i\rangle=
\left\{\begin{array}{l}
\large{2\cosh(d)(\rho^{-2}-\rho_0^{-2}) \textrm{\quad\quad\quad\quad}|X|<\rho}\\
\large{\frac{2\cosh(d)|X^{\perp}|^2}{|X|^{4}}-\frac{2\cosh(d)}
{\rho_0^{2}}\textrm{\quad}\rho<|X|<\rho_0}\\
\large{0 \textrm{\quad\quad\quad\quad\quad\quad\quad\quad\quad\quad}\rho_0<|X|}
\end{array}.\right.
\end{equation*}
Integrating the above equality on $M$ and substituting it to (\ref{L1}), we get
\begin{eqnarray}
\rho^{-2}\int_{M_{\rho}}\cosh{d} d\sigma_M-\rho_0^{-2}\int_{M_{\rho_0}}\cosh{d} d\sigma_M&=&-\int_{M_{\rho_0,\rho}}\frac{|X^{\perp}|^2}{|X|^4}\cosh{d}d\sigma_M \nonumber\\
&+&\int_{M_{\rho_0,\rho}}(|X|^{-2}-\rho_0^{-2})\langle X, H\nu\rangle d\sigma_M\nonumber\\
&+&\int_{M_\rho}(\rho^{-2}-\rho_0^{-2})\langle X, H\nu\rangle d\sigma_M ,\label{L3}
\end{eqnarray}
where $M_\rho\doteqdot\{x\in M : \sinh d(x)\leq \rho\}$ and $M_{\rho_0,\rho}\doteqdot M_{\rho_0}-M_{\rho}$.

For any $x\in M_{\rho_0, \rho}$, we have
\begin{equation}
-\cosh{d}\frac{|X^\perp|^2}{|X|^4}+\frac{\langle X, H\nu\rangle}
{|X|^2}-\rho_0^{-2}\langle X, H\nu\rangle\leq\frac{H^2}{4}.\label{L5}
\end{equation}
In fact, if $\langle X, H\nu\rangle\leq 0$, by $|X|^{-2}-\rho_0^{-2}\geq 0$ in  $M_{\rho_0, \rho}$,
we deduce the RHS of (\ref{L5}) is non-positive, and therefore (\ref{L5}) is true.
If $\langle X, H\nu\rangle>0$, by Cauchy-Schwartz inequality, we compute
\begin{equation}
-\cosh{d}\frac{|X^\perp|^2}{|X|^4}+\frac{\langle X, H\nu\rangle}{|X|^2}-\rho_0^{-2}\langle X, H\nu\rangle\leq-\frac{|X^\perp|^2}{|X|^4}+\frac{\langle X, H\nu\rangle}{|X|^2}
\leq\frac{H^2}{4}.
\end{equation}
Combining (\ref{L5}) with (\ref{L3}), we get
\begin{eqnarray*}
\rho^{-2}A(M_{\rho})&\leq&\rho_0^{-2}A(M_{\rho_0})+\int_{M_{\rho_0,\rho}}\frac{H^2}{4}d\sigma_M
+\int_{M_\rho}({\rho}^{-2}-\rho_0^{-2})\langle X, H\nu\rangle d\sigma_M\nonumber\\
&&+\int_{M_{\rho_0}}\frac{\cosh d-1 }{\rho^2_0}d\sigma_M\nonumber\\
&\leq&\rho_0^{-2}A(M_{\rho_0})+\int_{M_{\rho_0,\rho}}\frac{H^2}{4}d\sigma_M
+\int_{M_\rho}({\rho}^{-2}-\rho_0^{-2})\langle X, H\nu\rangle d\sigma_M\nonumber\\
&&+\frac{1}{2}A(M_{\rho_0}),
\end{eqnarray*}
where we have used that $\frac{\cosh d-1 }{\rho^2_0}\leq\frac{1}{2}$ in $M_{\rho_0}$.

Let $\rho\rightarrow 0$, we get
\begin{equation}
\pi\leq(\rho_0^{-2}+\frac{1}{2})A(M_{\rho_0})+\frac{1}{4}\int_{M_{\rho_0}}H^2d\sigma_M.\label{L2}
\end{equation}
Since $\{x \in M: \sinh d(x)<\rho_0\}\subset \{x \in M: d(x)<\rho_0\}$, we know that the inequality
(\ref{L2}) also holds as $M_{\rho_0}$ defined by $\{x \in M: d(x)<\rho_0\}$

Now we prove the Lemma 6.3 by
using (\ref{L2}). Let $\rho_0\doteqdot(\frac{A(M)}{W(M)})^{\frac{1}{2}}$
and choose an integer $N$
such that
\begin{equation}
2N\rho_0<diam(M)\leq 2(N+1)\rho_0. \label{L4}
\end{equation}
We can assume that $N\geq1$. Otherwise, we have
\begin{equation*}
diam(M)\leq 2\rho_0= 2(\frac{A(M)}{W(M)})^{\frac{1}{2}}\leq
\frac{1}{2\pi}(A(M)W(M))^{\frac{1}{2}},
\end{equation*}
here we have used the fact that  $\frac{W(M)}{4\pi}\geq 1$ which is deduced
from inequality (\ref{W3}).

For $N\geq 1$, we can choose $\{y_0,\cdots , y_N\}$ in $ M$, such that $d(y_i, y_j)\geq 2\rho_0$ for $i\neq j$ which implies the balls $M_{\rho_0}(y_i)\doteqdot\{x\in M: d(x,y_i)<\rho\}$ are disjoint. Using the formula (\ref{L2}) for every $M_{\rho_0}(y_i)$ and summing $i$ from $0$ to $N$, we derive
\begin{equation*}
(N+1)\pi\leq(\rho_0^{-2}+\frac{1}{2})A(M)+\frac{1}{4}W(M).
\end{equation*}
Combining with (\ref{L4}), we get
\begin{eqnarray*}
diam(M)&\leq&\frac{2\rho_0}{\pi}[(\rho_0^{-2}+\frac{1}{2})A(M)+\frac{W(M)}{4}]\nonumber\\
&=&\frac{2}{\pi}[(\frac{W(M)}{A(M)})^{\frac{1}{2}}+\frac{1}{2}
(\frac{A(M)}{W(M)})^{\frac{1}{2}}]A(M)+\frac{1}{2\pi}(A(M)W(M))^{\frac{1}{2}}\nonumber\\
&\leq&\frac{2}{\pi}[(\frac{W(M)}{A(M)})^{\frac{1}{2}}+\frac{1}{2}
(\frac{W(M)}{A(M)})^{\frac{1}{2}}]A(M)+\frac{1}{2\pi}(A(M)W(M))^{\frac{1}{2}}\nonumber\\
&=&\frac{7}{2\pi}A^{\frac{1}{2}}(M)W^{\frac{1}{2}}(M).
\end{eqnarray*}
In the third line, we have used the inequality $\frac{A(M)}{W(M)}\leq 1$ which is also deduced
from (\ref{W3}).
\end{proof}
In the following, we will give a lower
bound for the area of a dumbbell in terms of its length which must be satisfied for its
neck not to pinch off under MCF. Then by violating the lower bound, we can prove the theorem.
\begin{proof}{(\emph{of Theorem \ref{THE}})}
\quad Suppose $M_0$ is a topological sphere embedded in $\mathbb{H}^3$ which encloses
two round spheres $B$, both of radius $1$, with centres separated by a
distance $d>0$, then the diameter of $M_0$ is given by
\begin{equation*}
diam(M_0)\geq d+2.
\end{equation*}

By the comparison principle in Section 3 and equality (\ref{F6}),
we have the estimate of the diameter
\begin{equation}
diam(M_t)\geq d+2r(t), \label{M6}
\end{equation}
where $r(t)$ satisfies the following equation which comes from the MCF equation of round balls
\begin{equation*}
e^t \cosh r(t)=\cosh1,
\end{equation*}
for $t<T_0\doteqdot\ln {\cosh 1}$.

Combing (\ref{M5}) with (\ref{M6}), we get
\begin{equation*}
d+2r(t)\leq\frac{7}{2\pi}(A(M_t)W(M_t))^{\frac{1}{2}}=
\frac{7}{2\pi}[A(-\frac{1}{2}\frac{d}{dt}A)]^{\frac{1}{2}}=\frac{7}{4\pi}[-\frac{d}{dt}A^2(M_t)]^{\frac{1}{2}},
\end{equation*}
which implies
\begin{equation}
\frac{16\pi^2}{49}(d+2r(t))^2\leq -\frac{d}{dt}A^2(M_t).\label{M7}
\end{equation}
If there is non-singularity of $M_t$ before $T_0$, we have
\begin{equation*}
\frac{16\pi^2}{49}\int_0^{T_0}(d+2r)^2dt\leq A^2(M_0)-A^2(M_{T_0})\leq A^2(M_0),
\end{equation*}
which yields
\begin{equation}
A^2(M_0)\geq\frac{16\pi^2}{49}\int_0^{T_0}(d^2+4r^2+4rd)dt,\label{M8}
\end{equation}
in particular that
\begin{equation}
A^2(M_0)\geq\frac{16\pi^2T_0d^2}{49}.\label{M9}
\end{equation}
If (\ref{M9}) fails to hold, then a singularity must develop in the flow before time
\begin{equation}
t_0=\frac{49 A^2(M_0)}{16\pi^2d^2}.
\end{equation}

But we can choose $M_0$ to be the
shape of dumbell with a thin cylinder such that $A(M_0)$ satisfies
\begin{equation*}
A(M_0)=2Vol(B(1))+\epsilon d.
\end{equation*}
It is easy to see that the inequality (\ref{M9}) is
violated for some sufficiently large $d$, whenever $\epsilon$ is small enough. Thus such sphere $M_0$
must develop a singularity under MCF before the volume of it encloses shrinks to zero.
\end{proof}
\section{Acknowledgements}
The author would like to express his thanks to Professor L.
Ni and Professor X. Liu for brilliant guidance and stimulations. Thanks also go to A. Zhu and H.
Shao for a lot of conversations. The author would also like to thank P. Bryan for his interest in
this paper.

\end{document}